\documentclass[a4paper,12pt]{article}

\usepackage{amsmath,amssymb,amsthm}
\usepackage{latexsym}
\usepackage{graphics}
\usepackage{graphicx}
\usepackage{psfrag}

\theoremstyle{plain}
\newtheorem{Thm}{Theorem}[section]
\newtheorem{Lem}[Thm]{Lemma}
\newtheorem{Prop}[Thm]{Proposition}

\theoremstyle{definition}
\newtheorem*{Defi}{Definition}
\newtheorem{Rem}[Thm]{Remark}
\newtheorem{Ex}[Thm]{Example}

\title{The competition number of a generalized line graph is at most two}

\author{
\begin{tabular}{c}
{\sc Boram PARK}
\thanks{Corresponding author.}
\thanks{This work was supported by National Research Foundation of Korea Grant 
funded by the Korean Government, 
the Ministry of Education, Science and Technology. (NRF-2011-357-C00004).}\\
[1ex]
DIMACS, Rutgers University \\
Piscataway, NJ 08854, United States \\
{\tt boramp@dimacs.rutgers.edu}\\
\\
{\sc Yoshio SANO}
\thanks{The author was supported
by JSPS Research Fellowships for Young Scientists.}\\
[1ex]
National Institute of Informatics \\
Tokyo 101-8430, Japan \\
{\tt sano@nii.ac.jp}
\end{tabular}
}

\date{}

\begin{document}

\maketitle

\begin{abstract}
In 1982, Opsut showed that the competition number of a line graph
is at most two and gave a necessary and sufficient condition
for the competition number of a line graph being one.
In this note, we generalize this result to
the competition numbers of generalized line graphs,
that is, we show that the competition number of a generalized line graph
is at most two, and give necessary conditions and sufficient conditions
for the competition number of a generalized line graph being one.
\end{abstract}


\noindent
{\bf Keywords:}
competition graph;
competition number;
generalized line graph

\noindent
{\bf 2010 Mathematics Subject Classification:}
05C20, 05C76

\section{Introduction}

The notion of a competition graph was introduced by Cohen \cite{Cohen1}
as a means of determining the smallest dimension of ecological phase space.
The \emph{competition graph} $C(D)$ of a digraph $D$ is a graph
which has the same vertex set
as $D$ and an edge between two distinct vertices $u$ and $v$ if and only if
there exists a vertex $x$ in $D$ such that $(u,x)$ and $(v,x)$ are arcs of $D$.
Roberts \cite{MR0504054} observed that any graph $G$ together
with sufficiently many isolated vertices is the competition graph of
an acyclic digraph. The \emph{competition number} $k(G)$ of
a graph $G$ is defined to be the smallest nonnegative integer $k$
such that $G$ together with $k$ isolated vertices added
is the competition graph of an acyclic digraph.
It is not easy in general to compute $k(G)$
for an arbitrary graph $G$,
since Opsut \cite{MR679638} showed that the computation of the
competition number of a graph is an NP-hard problem.
It has been one of the important research problems in the study of
competition graphs to compute the exact values of
the competition numbers of various graphs.
For some special graph families, we have explicit formulae
for computing competition numbers:
If $G$ is a chordal graph without isolated vertices,
then $k(G)=1$; If $G$ is a nontrivial triangle-free connected graph
then $k(G)=|E(G)|-|V(G)|+2$ (\cite{MR0504054}).
For more recent results on graphs
whose competition numbers are calculated exactly,
see \cite{K, KS, KPS, KPS2012, KR, PKS2, PS1, PS2, S}.

The \emph{line graph} $L(H)$ of a graph $H$ is
the simple graph defined by $V(L(H))=E(H)$ and
$ee'\in E(L(H))$ if and only if $e$ and $e'$ have a vertex in common
and $e\neq e'$. 
A graph $G$ is called a \emph{line graph}
if there exists a graph $H$
such that $G$ is isomorphic to the line graph of $H$.
A \emph{clique} $S$ of a graph $G$ is a set of vertices of $G$
such that the subgraph induced by $S$ is a complete graph
(the empty set is also considered a clique).
A vertex $v$ in a graph $G$ is called \emph{simplicial} if
the neighborhood of $v$ in $G$ is a clique of $G$.
In 1982, Opsut \cite{MR679638} showed the following theorem.

\begin{Thm}[\cite{MR679638}]\label{thm:Opsut}
For a line graph $G$, $k(G)\le 2$ and the equality holds if and only if
$G$ has no simplicial vertex.
\end{Thm}

In this note, we investigate the competition number
of a generalized line graph
which was introduced by Hoffman \cite{Hoff} in 1970.
For a positive integer $m$,
the \emph{cocktail party graph} $CP(m)$
is the complete multipartite graph with $m$ partite sets
all of which have the size two:
\begin{eqnarray*}
V(CP(m)) &=& \bigcup_{l \in [m]} \{x_l, y_l \} \\
E(CP(m)) &=& \{x_i x_j \mid i,j \in [m], i < j \}
\cup \{y_i y_j \mid i,j \in [m], i < j \} \\
&& \cup \{x_i y_j \mid i,j \in [m], i \neq j \}
\end{eqnarray*}
where $[m]$ denotes the $m$-set $\{1,2,\ldots,m\}$.
Note that $CP(1)$ is the graph with two vertices and no edge. 
A \emph{vertex-weighted graph} $(H,f)$
is a pair of a graph $H$ and
a non-negative integer-valued function $f:V(H) \to \mathbb{Z}_{\geq 0}$
on the vertex set of $H$.
The \emph{generalized line graph} $L(H,f)$
of a vertex-weighted graph $(H, f)$ is the graph obtained from
the disjoint union of the line graph
$L(H)$ of the graph $H$ and
the cocktail party graphs $Q_v:=CP(f(v))$
where $v \in V(H)$ with $f(v)>0$
by adding edges between all the vertices in $Q_v$
and $e \in V(L(H))$ such that $e$ is incident to $v$ in $H$
(see Figure~\ref{fig1}).
For a graph $H$, if $f$ is a zero function
(i.e. $f(v)=0$ for any $v \in V(H)$),
then the generalized line graph of $(H,f)$
is the same as the line graph of $H$.
A graph $G$ is called a \emph{generalized line graph}
if there exists a vertex-weighted graph $(H, f)$
such that $G$ is isomorphic to the generalized line graph of $(H,f)$.

\begin{figure}
  \begin{center}
  \psfrag{a}{$v_1$}
  \psfrag{b}{$v_2$}
  \psfrag{c}{$v_3$}
  \psfrag{d}{$v_4$}
  \psfrag{e}{$v_6$}
  \psfrag{f}{$v_5$}
  \psfrag{A}{$Q_{v_1}$}
  \psfrag{E}{$Q_{v_5}$}
  \psfrag{C}{$Q_{v_3}$}
  \psfrag{H}{$H$}
  \psfrag{G}{$L(H,f)$}
  \includegraphics[width=13cm]{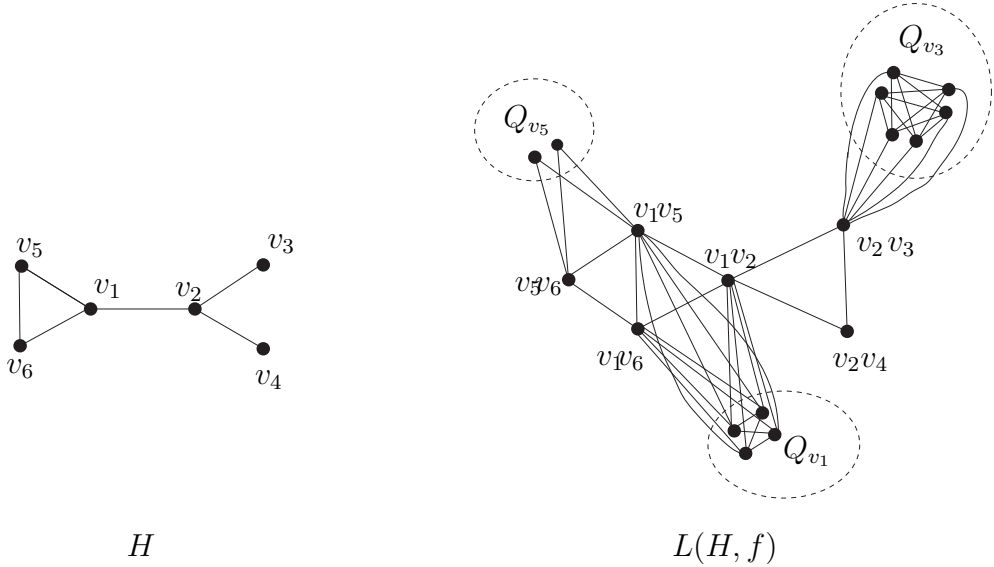}\\
  \caption{A vertex-weighted graph $(H,f)$,
  where $f:V(H) \to \mathbb{Z}_{\geq 0}$ is defined by
  $f(v_1)=2$, $f(v_3)=3$, $f(v_5)=1$, $f(v_2)=f(v_4) =f(v_6) =0$,
  and its generalized line graph $L(H,f)$}
  \label{fig1}
  \end{center}
\end{figure}

In this paper, we show the following result.

\begin{Thm}\label{thm:Main}
The competition number of a generalized line graph is at most two.
\end{Thm}

\noindent
This paper is organized as follows.
Section \ref{sec:2} is the main part of this paper.
Subsection \ref{subsec:2-1} gives some observations
on the competition graphs of acyclic digraphs
which will be used in this paper.
Subsection \ref{subsec:2-2} shows that the competition number
of a generalized line graph is at most two.
In Subsection \ref{subsec:2-3},
we investigate generalized line graphs whose competition numbers are one,
and give some sufficient conditions and necessary conditions.
Section \ref{sec:3} gives some concluding remarks.

\section{Main Results}\label{sec:2}
\subsection{Preliminaries}\label{subsec:2-1}

For a digraph $D$ and a vertex $v$ of $D$,
$N_D^+(v)$ and $N_D^-(v)$ denotes the out-neighborhood and
the in-neighborhood of $v$, respectively,
i.e., $N_D^+(v) := \{ x\in V(D)\mid (v,x)\in A(D) \}$
and $N_D^-(v) := \{ x\in V(D)\mid (x,v)\in A(D) \}$.
A digraph is said to be \emph{acyclic} if
it contains no directed cycles.
An ordering
$v_1, \ldots, v_{|V(D)|}$ of the vertices of
a digraph $D$
is called an \emph{acyclic ordering} of $D$
if $(v_i, v_j) \in A(D)$ implies $i<j$.
It is well-known that a digraph is acyclic if and only if
it has an acyclic ordering.

For a clique $S$ of a graph $G$
and an edge $e$ of $G$,
we say that \emph{$e$ is covered by $S$}
if both of the endvertices of $e$ are contained in $S$.
An \textit{edge clique cover} of a graph $G$
is a family of cliques of $G$ such that
each edge of $G$ is covered by some clique in the family.
The \emph{edge clique cover number} $\theta_E(G)$
of a graph $G$
is the minimum size of an edge clique cover of $G$.
A \textit{vertex clique cover} of a graph $G$
is a family of cliques of $G$ such that each vertex of $G$
is contained in some clique in the family.
The \textit{vertex clique cover number} $\theta_V(G)$
of a graph $G$
is the minimum size of a vertex clique cover of $G$.
For a graph $G$ and a vertex $v$ of $G$,
$\theta_V(N_G(v))$ is the vertex clique cover number of
the subgraph of $G$ induced by the neighbors of $v$ in $G$.
Opsut \cite{MR679638} showed the following lower bound
for the competition number of a graph 
(see also \cite{S12} for a generalization of this inequality). 

\begin{Prop}[\cite{MR679638}]\label{opsut:lower}
For any graph $G$,
$k(G) \geq \min_{v \in V(G)} \theta_V(N_G(v))$.
\end{Prop}

For a positive integer $k$,
we denote by $I_k$ the edgeless graph on $k$ vertices,
i.e., the graph with $k$ vertices and no edges.
The following lemma which comes from an easy observation
is elementary but useful.

\begin{Lem}\label{lem:1}
Let $G$ be a graph with at least two vertices
and let $k$ be an integer such that $k\ge k(G)$.
Then there exists an acyclic digraph $D$
such that
\begin{itemize}
\item[{\rm (a)}]
$C(D)=G \cup I_k$,
\item[{\rm (b)}]
$D$ has an acyclic ordering
$v_1, \ldots, v_{|V(G)|}, v_{|V(G)|+1}, \ldots, v_{|V(G)|+k}$,
where $V(G)=\{v_1,  \ldots,$ $v_{|V(G)|} \}$
and $V(I_k)= \{v_{|V(G)|+1}, \ldots, v_{|V(G)|+k} \}$, and
\item[{\rm (c)}]
$N_{D}^{-}(v_1) = N_{D}^{-}(v_2) = \emptyset$.
\end{itemize}
\end{Lem}

\begin{proof}
By the definition of the competition number of a graph,
there exists an acyclic digraph $D_1$ satisfying (a).
Let $D_2$ be the digraph obtained from $D_1$ by deleting
all the arcs outgoing from any vertices in $I_k$.
Then we can check that $D_2$ is an acyclic digraph satisfying (a) and (b).
Since there is no arc outgoing from any vertex in $I_k$,
there is an acyclic ordering
$v_1, v_2, \ldots, v_{|V(G)|},$ 
$v_{|V(G)|+1}, \ldots, v_{|V(G)|+k}$
of $D_2$ such that $V(G)=\{v_1, v_2, \ldots, v_{|V(G)|}\}$
and $V(I_k)= \{v_{|V(G)|+1}, \ldots, v_{|V(G)|+k} \}$.
By the definition of an acyclic ordering of a digraph,
it holds that
$N_{D_2}^{-}(v_1)=\emptyset$
and $N_{D_2}^{-}(v_2) \subseteq \{v_1\}$.
If $N_{D}^{-}(v_2) = \{ v_1\}$, then let $D$ be the digraph
obtained from $D_2$ by deleting the arc $(v_1,v_2)$.
Otherwise, let $D=D_2$.
Then $D$ is an acyclic digraph satisfying (a), (b), and (c).
Thus the statement holds.
\end{proof}

The competition number of a cocktail party graph
is given by Kim, Park and Sano \cite{KPS2012}.

\begin{Prop}[\cite{KPS2012}]\label{thm:KPS}
The competition number of a cocktail party graph $CP(m)$
with $m \geq 2$
is equal to two.
\end{Prop}

\subsection{Proof of Theorem \ref{thm:Main}}\label{subsec:2-2}

In this subsection,
we show that the competition number of a generalized line graph
is at most two.

For two vertex-disjoint graphs $G$ and $H$ and a clique $K$ of $G$,
we define the graph $G \ltimes_K H$
by
\begin{eqnarray*}
V(G \ltimes_K H) &:=& V(G)\cup V(H) \\
E(G \ltimes_K H) &:=& E(G)\cup E(H) \cup \{ uv\mid u\in K, v\in V(H)\}
\end{eqnarray*}
(see Figure~\ref{GKH}).

\begin{figure}
  \begin{center}
  \psfrag{G}{$G$}
  \psfrag{K}{$K$}
  \psfrag{H}{$H$}
  \includegraphics{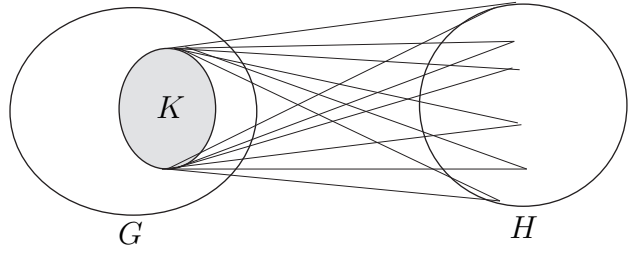}\\
  \caption{$G\ltimes_K H $}\label{GKH}
  \end{center}
\end{figure}

From Lemma \ref{lem:1},
we introduce a notion of top-two of a graph as follows.

\begin{Defi}
{\rm
For a graph $G$, a set $\{u,v\}$ of two distinct vertices $u$ and $v$ of $G$
is called a \emph{top-two} of $G$
if there exists an acyclic digraph $D$
such that $C(D) = G \cup I_{k(G)}$ and
$D$ has an acyclic ordering whose first and second vertices
are $u$ and $v$.
}
\qed
\end{Defi}

Note that any
graph
with at least two vertices always
has at least one top-two.
If a graph $G$
has no edges,
then any pair of two vertices of $G$
is a top-two of $G$.

\begin{Prop}\label{prop:GHK}
Let $G$ and $H$ be graphs with at least two vertices
such that $V(G) \cap V(H) = \emptyset$
and let $K$ be a clique of $G$.
Suppose that there exists an acyclic digraph $D'$
such that $C(D') = G \cup \{u_1,u_2\}$ 
where $\{u_1,u_2\}$ is a top-two of $H$.
If either $H$ has no edges or $H$
has no isolated vertices,
then there exists an acyclic digraph $D$ satisfying the following:
\begin{itemize}
\item[{\rm (i)}]
$C(D) = \left\{
\begin{array}{ll}
(G \ltimes_K H)  \cup I_{2}
& \text{ if } H \text{ has no edges}, \\
(G \ltimes_K H ) \cup I_{|k(H)|}
& \text{ if } H \text{ has no isolated vertices},
\end{array}
\right.$
and
\item[{\rm (ii)}]
$D$ has an acyclic ordering
whose first $|V(G)|+2$ terms induce the digraph $D'$.
\end{itemize}
\end{Prop}

\begin{proof}
Suppose that $H$ has no edges.
Let $V(H)=\{u_1,u_2,\ldots,u_m\}$, where $m \geq 2$.
We define a digraph $D$ by
$V(D) := V(G) \cup V(H) \cup \{u_{m+1},u_{m+2} \}$ and
\[
A(D) := A(D') \cup \left( \bigcup_{i=1}^{m}
\{ (x,u_{i+2}) \mid x \in K \cup \{u_i\} \} \right)
\]
where $u_{m+1}$ and $u_{m+2}$ are 
new 
vertices.
Then, it is easy to see that $D$ is an acyclic digraph
satisfying (i) and (ii).

Suppose that $H$ has no isolated vertices.
Let $k:=k(H)$.
By the assumption that $\{u_1,u_2\}$ is a top-two of $H$,
there exists an acyclic digraph $D''$
such that $C(D'') = H \cup I_k$ and
the two vertices $u_1$ and $u_2$ in $D''$ satisfy
$N_{D''}^{-}(u_1)=N_{D''}^{-}(u_2) = \emptyset$.
Let $D$ be the digraph defined by $V(D) := V(D') \cup V(D'')$ and
$A(D) := A(D') \cup A(D'') \cup A^*$
where
\[
A^* := \{ (u,v) \mid u \in K, v \in V(D'') \setminus \{u_1,u_2\} \}.
\]
Since the ordering obtained
by attaching an acyclic ordering of $D''$ at the end of
an acyclic ordering of $D'$ gives an acyclic ordering of $D$,
$D$ is acyclic and satisfies (ii).
Since $H$ has no isolated vertices,
each vertex in $H$ has an out-neighbor in $D''$.
Therefore, each edge of $G \ltimes_K H$
between
a vertex in $K$
and a vertex of $H$
is an edge of $C(D)$ which
results from $A^*$. 
Thus $C(D) = (G \ltimes_K H) \cup I_k$ and so $D$ satisfies (i).
\end{proof}

For any vertex $v$ of a graph $H$, let $K_H(v)$ denote
the set of the edges incident to $v$ in $H$.
Note that $K_H(v)$ forms a clique of the line graph of $H$
for each vertex $v$
in a graph $H$ and $\{ K_H(v) \mid v \in V(H) \}$
is an edge clique cover
of the line graph of $H$.

\begin{Thm}\label{lem:main}
Let $(H,f)$ be a vertex-weighted graph.
For any edge $e=uv$ of $H$,
there exists an acyclic digraph $D$ such that
$C(D) = L(H,f) \cup I_2$ and
that
$N_{D}^-(z_1)=K_H(u)$ and $N_{D}^-(z_2)=K_H(v)$,
where $V(I_2)=\{z_1,z_2\}$.
\end{Thm}

\begin{proof}
Let $G := L(H,f)$ for convenience.
First, we consider the case where $f$ is a zero function.
We show the theorem by induction on the number of edges of $H$.
If $H$ has at most one edge, then the statement is checked easily.
Assume that the statement is true for any graph with $m-1$ edges
where $m \geq 2$.
Let $H$ be a graph with $m$ edges.
It is sufficient to consider the case where $H$ is connected.
Take an edge $e=uv$ of $H$.
Since $m \geq 2$ and $H$ is connected,
there exists an edge $e'$
incident
to $e$ in $H$.
Without loss of generality,
we may assume that
the vertex $u$ is also an endvertex of $e'$.
Let $H'$ be the graph obtained from $H$ by deleting the edge $e$.
Then $L(H')$ is the graph obtained from $L(H)$ by deleting the vertex $e$.
Since $H^\prime$ has $m-1$ edges,
by the induction hypothesis,
there exists an acyclic digraph $D'$ such that
$C(D') = L(H^\prime) \cup \{z_1, e \}$
and $N_{D^\prime}^-(z_1)=K_{H^\prime}(u)$.

Now we define a digraph $D$ by
$V(D) := V(D') \cup \{z_2\} = V(L(H)) \cup \{z_1,z_2\}$ and
\[
A(D) := A(D') \cup \{ (e,z_1) \} \cup \{(e'',z_2) \mid e'' \in K_H(v)\}.
\]
Then the ordering of the vertices in $D$
obtained from an acyclic ordering of $D'$
by adding the vertex $z_2$ to it as the last term
is an acyclic ordering of $D$, and so $D$ is acyclic.
By the definitions of $D$ and $H'$, $N_D^-(z_2)=K_{H}(v)$ and
$N_{D}^-(z_1) = N_{D^\prime}^-(z_1) \cup \{ e \}
= K_{H'}(u) \cup \{ e \} = K_{H}(u)$.
It is easy to see that $C(D)=L(H)\cup\{z_1,z_2\}$.
Thus the theorem holds when $f$ is a zero function.

Second, we consider the case where
$f$ is not a zero function.
Let $v_1, v_2, \ldots, v_n$ be the vertices of $H$ such that $f(v_i)>0$.
For each $i \in \{0\} \cup [n]$,
we define a graph $G_i$
by
\[
G_0 := L(H) \quad
\text{ and } \quad
G_i := G_{i-1} \ltimes_{K_H(v_i)} Q_{v_i}
\quad (i \in [n]),
\]
where $Q_{v_i}=CP(f(v_i))$. Note that $G_n=G$.

Take an edge $e=uv$ of $H$.
Since $f(v_1)\neq 0$, $Q_{v_1}$ has at least two vertices.
Then $Q_{v_1}$ has a top-two $\{z_1,z_2\}$.
Since $G_0$ is a line graph,
by the above argument on the case of $f=0$,
there exists an acyclic digraph $D_0$ such that
$C(D_0)=G_0 \cup \{ z_1,z_2\}$ and $N_{D_0}^-(z_1)=K_H(u)$
and $N_{D_0}^-(z_2)=K_H(v)$.

Since $Q_{v_1}$
has no edges or $Q_{v_1}$ is connected,
it follows from
Propositions \ref{thm:KPS} and \ref{prop:GHK} that
there exists an acyclic digraph $D_1$
such that $C(D_1) = G_1 \cup I_{2}$
and
$D_1$ has an acyclic ordering whose first $|V(G_0)|+2$ terms
induce the digraph $D_0$.
Then $N_{D_1}^-(z_1) = N_{D_0}^-(z_1) = K_H(u)$
and $N_{D_1}^-(z_2) = N_{D_0}^-(z_2) = K_H(v)$.
From repeating the process, we can obtain an acyclic digraph $D_n$
such that $C(D_n)=G_n\cup I_2$ and $N_{D_n}^-(z_1)=K_H(u)$
and $N_{D_n}^-(z_2)=K_H(v)$.
Let $D:=D_n$.
Since $G_n=G$,
the theorem holds.
\end{proof}

\begin{proof}[Proof of Theorem~\ref{thm:Main}]
It immediately follows
from Theorem~\ref{lem:main}
that the competition number of
a generalized line graph is at most two.
\end{proof}

\subsection{Generalized line graphs with competition number one}
\label{subsec:2-3}

In the following, we
show some
necessary conditions and sufficient conditions
for the competition number of a connected generalized line graph being one.
Theorem~\ref{thm:Opsut} says that
the competition number of a connected line graph is one
if and only if it has a simplicial vertex.

Since the case of the generalized line graph $L(H,f)$
of a vertex-weighted graph $(H,f)$ with a zero function $f$
is reduced to Theorem \ref{thm:Opsut},
we consider the case
where $f$ is
a nonzero function.
In this subsection, when we consider the generalized line graph
$L(H,f)$
of a vertex-weighted graph $(H,f)$,
we denote
the cocktail party graph $CP(f(v))$ added to $L(H)$
by $Q_v$
for each vertex $v$ of $H$
in cases where this notation will not cause confusion.

First, we give necessary conditions for the competition number
of a connected generalized line graph being one.

\begin{Lem}\label{lem:one-simp}
If a graph $G$ has competition number one,
then $G$ has a simplicial vertex or an isolated vertex.
\end{Lem}

\begin{proof}
If $G$ has no simplicial vertex and no isolated vertex,
then the competition number of $G$ is at least two
by Proposition \ref{opsut:lower}.
Thus the lemma holds.
\end{proof}

\begin{Thm}\label{thm:Main2}
Let $G$ be the generalized line graph of
a connected vertex-weighted graph $(H,f)$ with a nonzero function $f$.
If $k(G)=1$, then at least one of the following holds:
\begin{itemize}
\item[{\rm (i)}]
$f(v)=1$ for some vertex $v$ of $H$,
\item[{\rm (ii)}] There exists a vertex $v$ of $H$
such that $f(v)=0$
and $K_H(v)$ contains a simplicial vertex of $G$.
\end{itemize}
\end{Thm}

\begin{proof}
Assume that $k(G)=1$.
Suppose that (i) does not hold.
Then $f(u)\neq 1$ for any vertex $u \in V(H)$.
As $k(G)=1$ and $G$ is connected,
$G$ has a simplicial vertex $x$ by Lemma \ref{lem:one-simp}.
Since $V(G) = \bigcup_{u \in V(H)}\left( K_H(u) \cup Q_u \right)$,
the simplicial vertex
$x$ is contained in $K_H(v) \cup Q_v$ for some $v\in V(H)$.
Since $f(v) \neq 1$,
either $f(v)=0$ or $f(v)\ge 2$.
If $f(v) \geq 2$,
then any vertex in $K_H(v) \cup Q_v$
is not simplicial, which is a contradiction.
Therefore $f(v)=0$.
Thus the simplicial vertex $x$ is contained in $K_H(v)$.
Hence (ii) holds.
\end{proof}

\begin{Rem}
{\rm
The conditions (i) and (ii) in Theorem~\ref{thm:Main2}
are not sufficient conditions for generalized line graphs
to have competition number one.
}
\qed
\end{Rem}

\begin{figure}
  \begin{center}
  \psfrag{a}{$v_1$}
  \psfrag{b}{$v_2$}
  \psfrag{c}{$v_3$}
  \psfrag{d}{$v_4$}
  \psfrag{H}{$H$}
  \psfrag{P}{$Q_{v_3}$}
  \psfrag{Q}{$Q_{v_4}$}
  \psfrag{G}{$L(H,f)$}
  \includegraphics{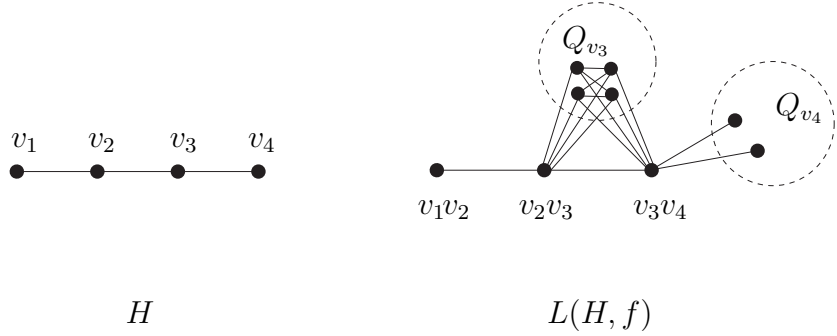}\\
  \caption{A vertex-weighted graph $(H,f)$,
  where $f:V(H) \to \mathbb{Z}_{\geq 0}$ is defined by
  $f(v_1)=f(v_2)=0$, $f(v_3)=2$, $f(v_4)=1$,
  and its generalized line graph $L(H,f)$}
  \label{fig:ex}
  \end{center}
\end{figure}

\begin{Ex}
{\rm
Let $(H,f)$ be the vertex-weighted graph 
where $H$ is the graph defined by 
$V(H)=\{ v_1, v_2,$ $v_3, v_4 \}$ 
and $E(H)=\{ v_1v_2, v_2v_3, v_3v_4\}$ 
and 
$f:V(H) \to \mathbb{Z}_{\geq 0}$ is 
defined by $f(v_1)=f(v_2)=0$, $f(v_3)=1$,
$f(v_4) =2$ 
(see Figure~\ref{fig:ex}).
Then, the generalized line graph of $(H,f)$
satisfies both (i) and (ii) of Theorem~\ref{thm:Main2}. 
But, the competition number of $L(H,f)$ is two. 

To see this, we recall a result by Kim \cite{K} 
which states that 
the deletion of some pendant vertices from a connected graph 
does not change its competition number 
if the resulting graph has at least two vertices. 
Let $G'$ be 
the graph obtained from $L(H,f)$ 
by deleting the two vertices in $Q_{v_4}$ and the vertex $v_1v_2$. 
Then, $k(G') = k(L(H,f))$. 
Since $G'$ contains neither an isolated vertex nor a simplicial vertex, 
$k(G')\geq 2$ by Lemma~\ref{lem:one-simp}. 
By Theorem \ref{thm:Main}, $k(L(H,f)) = 2$.}
\qed
\end{Ex}

Next, we show the following result
which gives sufficient conditions for generalized line graphs
to have competition number one.

\begin{Thm}\label{thm:Main3}
Let $G$ be the generalized line graph of
a connected vertex-weighted graph $(H,f)$ where $H$ has
at least one edge and $f$ is a nonzero function.
Then, $k(G)=1$ if one of the following holds:
\begin{itemize}
\item[{\rm (i)}]
$f(u)=f(v)=1$ for some edge $e=uv$ of $H$,
\item[{\rm (ii)}]
$f(v) \leq 1$ for any vertex $v$ of $H$.
\end{itemize}
\end{Thm}

\begin{proof}
Note that $G$ is a connected graph with at least two vertices,
since $H$ is a connected graph with at least one edge and $f$
is a nonzero function.
Therefore $k(G)\ge 1$. It is sufficient to show that
there exists an acyclic digraph $D$ such that $C(D)=G\cup I_1$.

Suppose that (i) holds. Then,
$f(u)=f(v)=1$ for some edge $uv$ of $H$
and therefore both $Q_u$ and $Q_v$ are
the edgeless graph $I_2$ on two vertices.
Let $V(Q_{u})=\{ q_{u}, q'_{u} \}$
and $V(Q_{v})=\{ q_{v}, q'_{v}\}$.
Let $f_0:V(H) \to \mathbb{Z}_{\geq 0}$
be the function defined by $f_0(u)=f_0(v)=0$
and $f_0(x)=f(x)$ for any $x\in V(H)\setminus\{u,v\}$.
Let $G_0$ be the generalized line graph of $(H,f_0)$.
Then $G_0$ is isomorphic to the graph obtained from $G$
by deleting the four vertices $q_u$, $q'_u$, $q_v$, and $q'_v$.
By Theorem~\ref{lem:main},
$k(G_0)\le 2$ and there exists an acyclic digraph $D_0$
such that $C(D_0)=G_0\cup \{ q_v, z\}$,
where $z$ is a new vertex,
and $K_H(u) = N_{D_0}^{-}(q_v)$
and $K_H(v) = N_{D_0}^{-}(z)$.
Let $D$ be a digraph defined by $V(D) := V(G) \cup \{ z \}$ and
\begin{eqnarray*}
A(D) &:=& A(D_0)
\cup \{ (q_u, q'_u), (q'_u, q_v), (q_v, q'_v), (q'_v, z) \} \\
&&
\cup \{ (x, q'_u) \mid x \in K_H(u) \}
\cup \{ (x, q'_v) \mid x \in K_H(v) \}.
\end{eqnarray*}
Then $D$ is acyclic
and $C(D) = G \cup \{ z \}$.
Therefore $k(G) \leq 1$.

Next, suppose that (ii) holds.
Then $f(v) \leq 1$ for any vertex $v$ of $H$.
Let $S := \{ v \in V(H) \mid f(v) =1 \}$.
Since $f$ is not a zero function,
the set
$S$ is not empty.
Let $S:=\{ u_1,u_2,\ldots,u_t\}$, where $t:=|S|>0$.
Then $Q_{u_i}=CP(1)=I_2$ and
let $V(Q_{u_i})=\{ q_{2i-1}, q_{2i}\}$
for $i \in [t]$.
By Theorem~\ref{lem:main},
there exists an acyclic digraph $D_0$
such that $C(D_0) = L(H) \cup \{ z, q_{2t}\}$
where $z$ is a new vertex and $N_{D_0}^{-}(z) = K_H(u_1)$.
Let $v_1,v_2,\ldots,v_n,q_{2t},z$ be an acyclic ordering of $D_0$.
If $t=1$ then let $R := \emptyset$, and
if $t>1$ then let
\[
R := \bigcup_{i=1}^{t-1} \{ (x, q_{2i}) \mid
x \in K_H(u_{i+1}) \cup \{ q_{2i+1} \} \}.
\]
Let $D$ be the digraph defined by
$V(D) := V(D_0) \cup \{q_1, \ldots, q_{2t-1} \}
= V(G) \cup \{ z \}$ and
\[
A(D) 
:= 
A(D_0) \cup \{ (q_1,z)\} 
\cup \left( \bigcup_{i=1}^{t}
\{ (x, q_{2i-1}) \mid x \in K_H(u_i) \cup \{ q_{2i} \} \} \right) \cup R.
\]
Then $D$ is acyclic since the ordering
$v_1$, $v_2$, \ldots, $v_n$, $q_{2t}$, $q_{2t-1}$, $q_{2}$, $q_1$, $z$
of the vertices of $D$ is an acyclic ordering of $D$.
In addition, 
it follows from the definitions of $D$ and $G$ that 
$C(D)=G \cup \{z\}$.
Therefore $k(G) \leq 1$.

Thus $k(G)=1$, and hence the theorem holds. 
\end{proof}

\begin{Rem}
{\rm
Each of the conditions (i) and (ii) in Theorem~\ref{thm:Main3}
is not necessary conditions for generalized line graphs
to have competition number one.}
\qed
\end{Rem}

\begin{figure}
  \begin{center}
  \psfrag{a}{$v_1$}
  \psfrag{b}{$v_2$}
  \psfrag{c}{$v_3$}
  \psfrag{d}{$v_4$}
  \psfrag{H}{$H$}
  \psfrag{Q}{$Q_{v_4}$}
  \psfrag{G}{$L(H,f)$}
  \includegraphics{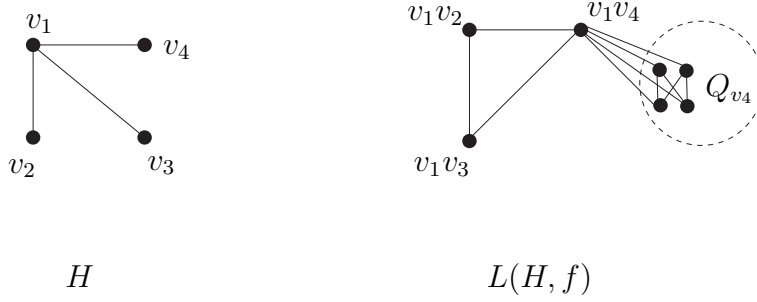}\\
  \caption{A vertex-weighted graph $(H,f)$,
  where $f:V(H) \to \mathbb{Z}_{\geq 0}$ is defined by
  $f(v_1)=f(v_2) =f(v_3) =0$ and $f(v_4)=2$,
  and its generalized line graph $L(H,f)$}
  \label{fig:ex2_8}
  \end{center}
\end{figure}

\begin{Ex}
{\rm
Let $(H,f)$ be the vertex-weighted graph
where $H$ is the graph defined by
$V(H)=\{ v_1, v_2, v_3, v_4\}$
and $E(H)=\{ v_1v_2, v_1v_3, v_1v_4 \}$
and $f:V(H) \to \mathbb{Z}_{\geq 0}$ is the function
defined by $f(v_1)=f(v_2)=f(v_3)=0$ and  $f(v_4)=2$
(see Figure~\ref{fig:ex2_8}).
Then, the generalized line graph of $(H,f)$
has the competition number one
but $(H,f)$ satisfies
neither (i) nor (ii) of Theorem~\ref{thm:Main3}.

Let $V(Q_{v_4})=\{q_1,q_2,q_3,q_4\}$
and $E(Q_{v_4})=\{q_1q_2,q_2q_3,q_3q_4,q_4q_1\}$.
To see $k(L(H,f))=1$, we define
a digraph $D$ by
\begin{eqnarray*}
V(D) &:=& V(L(H,f)) \cup \{z\} = E(H) \cup V(Q_{v_4}) \cup \{z\},  \\
A(D) &:=& 
\{(q_1,q_2),(q_4,q_2),(v_1v_2,q_2)\} 
\cup \{(q_1,q_3),(q_2,q_3),(v_1v_2,q_3) \} \\
&& \cup \{(q_2,v_1v_3),(q_3,v_1v_3),(v_1v_2,v_1v_3) \} \\
&& \cup \{(q_3,v_1v_4),(q_4,v_1v_4),(v_1v_2,v_1v_4) \} \\
&& \cup \{(v_1v_2,z), (v_1v_3,z),(v_1v_4,z) \} 
\end{eqnarray*}
where $z$ is a new vertex.
Then we can easily check that $C(D) = L(H,f) \cup \{z\}$,
and that $D$ is acyclic 
since the ordering 
$q_1, q_4, v_1v_2, q_2, q_3, v_1v_3, v_1v_4, z$ 
is an acyclic ordering of $D$. 
Therefore $k(L(H,f)) \leq 1$,
Since $L(H,f)$ is connected,
$k(L(H,f)) \geq 1$.
Hence  $k(L(H,f)) = 1$.
}
\qed
\end{Ex}

\section{Concluding Remark}\label{sec:3}

In this note, we showed that the competition number of
a generalized line graph is at most two,
which is an extension of a result on the competition number of a line graph.
In addition, we tried to characterize
generalized line graphs whose competition numbers are one,
and then found necessary conditions and sufficient conditions
for the competition number of a generalized line graph being one.
However, these conditions are not necessary-and-sufficient conditions,
so it still remains open to give a complete characterization of
generalized line graphs whose competition numbers are one.



\end{document}